\documentclass[preprint]{imsart}

\usepackage{graphicx,epsf}
\RequirePackage[OT1]{fontenc}
\usepackage{graphicx,epsf}
\RequirePackage{amsthm,amsmath}
\RequirePackage[numbers]{natbib}
\RequirePackage[colorlinks,citecolor=blue,urlcolor=blue]{hyperref}
\usepackage{graphicx,epsf}

\usepackage{amsfonts}
\usepackage{amsmath}
\usepackage{amsthm}
\usepackage{amssymb}
\usepackage{latexsym}

\theoremstyle{plain}
  \newtheorem{theorem}{Theorem}
  \newtheorem{lemma}{Lemma}

\theoremstyle{definition} % For roman text in the body
    
    \newtheorem{remark}{Remark}

 \theoremstyle{remark} % For an italic header, more subtle than definition style

 \newcommand{\E}{\mbox{\rm \hspace*{.2ex}I\hspace{-.5ex}E\hspace*{.2ex}}}

\newcommand{\Vol}{\operatorname{Vol}}

\newcommand{\sm}{{\raise0.3ex\hbox{$\scriptstyle \setminus$}}}

\startlocaldefs
\theoremstyle{plain}

\endlocaldefs

\begin{document}

\begin{frontmatter}
\title{On sequential maxima of exponential sample means, with an application
to ruin probability
 %\thanksref{T1}
 }
\runtitle{Maxima of exponential sample means}
%\thankstext{T1}{Footnote to the title with the ``thankstext'' command.}

\begin{aug}
\author{\fnms{Dimitris} \snm{Cheliotis}\ead[label=e1]{dcheliotis@math.uoa.gr, npapadat@math.uoa.gr}}
 %,t2,m1}
\and
\author{\fnms{Nickos} \snm{Papadatos}\ead[label=e2]{npapadat@math.uoa.gr}}
%\thanksref{t3,m1,m2}
%\and
%\author{\fnms{Third} \snm{Author}\thanksref{t1,m2}
%\ead[label=e3]{third@somewhere.com}
%\ead[label=u1,url]{http://www.foo.com}}
 %\thankstext{t1}{Corresponding author}
 %\thankstext{t2}{First supporter of the project}
 %\thankstext{t3}{Second supporter of the project}
 \runauthor{D.\ Cheliotis and N.\ Papadatos}

 \affiliation{National and Kapodistrian University of Athens
 %\thanksmark{m1}
 }
 %and Another University\thanksmark{m2}}

 \address{National and Kapodistrian University of Athens
 \\
 Department of Mathematics
 \\
 Panepistemiopolis
 \\
 GR-157 84 Athens\\
 Greece\\
 \printead{e1}\\
 \phantom{E-mail:\ dcheliotis@math.uoa.gr, npapadat@math.uoa.gr}
 %\printead*{e2}
 }
\end{aug}

 \begin{abstract}
 We obtain the distribution of the maximal average in a sequence
 of independent identically distributed exponential random variables.
 Surprisingly enough,
 it turns out that the inverse distribution admits a simple closed form.
 An application to ruin probability in a risk-theoretic model
 is also given.
 \end{abstract}



 \end{frontmatter}

 \section{Introduction}
 \setcounter{equation}{0}
 \label{sec.1}

 Consider a sequence $(X_i)_{i\ge1}$ of independent identically distributed
 (i.i.d.)\ random variables,
 each having exponential distribution with mean 1. For each $i\in\mathbb{N}^+$
 define the sample mean of the
 first $i$ variables as $\overline{X}_i:=(X_1+X_2+\cdots+X_i)/i$. The
 supremum of this sequence
 $$Z_\infty:=\sup\{\bar{X}_i: i\in \mathbb{N}^+\}$$
 is finite because the sequence converges to 1 with probability 1.

 In this note we compute the distribution function, $F_\infty$, of $Z_\infty$.
 In fact, what has nice form is the inverse of this distribution function.
 Our main result is the following.

 \begin{theorem}
 \label{theo.2.1}
 {\rm (a)}
 $Z_\infty$ has distribution function
 \begin{equation*}
 %\label{DistrFunction}
 F_\infty(x)=1-\sum_{k=1}^\infty \frac{k^{k-1}}{k!} x^{k-1}e^{-kx}
 \end{equation*}
 for $x>0$, and density
 which is continuous on $\mathbb{R}\sm\{1\}$, positive on $(1, \infty)$, and zero on $(-\infty, 1)$.

\noindent  {\rm (b)}
 The restriction of $F_\infty$ on $(1, \infty)$ is one to one and onto
 $(0, 1)$ with inverse
 \begin{equation}
 \label{InverseDF}
 F_{\infty}^{-1}(u)=\frac{-\log(1-u)}{u} \ \  \text{ for all } u\in(0, 1).
 \end{equation}
 \end{theorem}

 \begin{remark}

 \noindent (a) For $F_\infty$ we have the alternative expression
 $$F_\infty(x)=1+\frac{1}{x} W_0(-xe^{-x})$$
 where $W_0$ is the principal
 branch of the Lambert $W$ function, that is, the inverse function
 of $x\mapsto x e^{x}, x\ge1$; see \cite{CGHJK}. Indeed, the power series $\sum_{k=1}^\infty \frac{k^{k-1}}{k!} y^k$ has interval of convergence $[-1/e, 1/e]$ and equals $-W_0(-y)$.

 (b) Clearly, the results of the theorem extend immediately to the case
 that the $X_i$'s are i.i.d. and $X_1=a Y+b$
 with $a>0$, $b\in \mathbb{R}$ and $Y\sim {\mbox Exp}(1)$. However, we were not able to
 find an explicit formula for the distribution of $Z_\infty$ for any other
 distribution of the  $X_i$'s.

 (c) Although it is intuitively
 clear that $F_{\infty}(x)>0$ for $x>1$, it is not entirely
 obvious how to verify it by direct calculations.
 However, this fact is evident from Theorem \ref{theo.2.1}.

 (d) Formula
 \eqref{InverseDF} enables the explicit calculation of the percentiles
 of $F_{\infty}$. Therefore, the
 result is useful for the following kind of problems:
 Suppose that a quality control machine calculates subsequent averages,
 and alarms if
 some average $\bar{X}_n$ is greater than $c$,
 where $c$ is a predetermined constant
 such that the probability of false alarm is small, say
 $\alpha$. For $\alpha\in(0,1)$, the upper percentage
 point of $F_{\infty}$ (that is,  the point $c_\alpha$
 with $F_{\infty}(c_\alpha)=1-\alpha$)
 is given by $c_{\alpha}=\frac{-\log\alpha}{1-\alpha}$,
 and thus the proper value of $c$ is $c=c_{\alpha}$.
 \end{remark}

 If in the definition of $Z_\infty$ we discard the first $n-1$ values
 of  $\bar{X_i}$, we obtain the random variable
 \begin{equation*}
 M_n:=\sup\{\bar{X}_i: i\ge n\}
 \end{equation*}
 for which, however, (for $n\ge2$) the distribution function
 is quite complicated even for the exponential case.
 For instance, the distribution of $M_2$ is given by
 (we omit the details)
 \begin{equation*}
 F_{M_2}(x)= F_\infty(x)+e^{-2x}\frac{F_\infty(x)}{1-F_\infty(x)},
  \ \ \
  x\geq 0.
 \end{equation*}
 What we can compute is the asymptotic distribution of $\sqrt{n}(M_n-1)$ as
 $n\to\infty$. This distribution is the same for a large class of
 distributions of the $X_i$'s, as the following theorem shows.

 \begin{theorem}
 \label{WeakConvThm}
 Assume that the $(X_i)_{i\ge1}$ are i.i.d.\ with mean 0, variance 1,
 and there is $p>2$ with $\E|X_1|^p<\infty$. Let
 $M_n:=\sup\{\bar{X}_i: i\ge n\}$ for all $n\in\mathbb{N}^+$.
 Then,
 \begin{equation*}
 \sqrt{n}M_n\Rightarrow |Z|
 \end{equation*}
 where $Z\sim N(0, 1)$ is a standard normal random variable.
 \end{theorem}

 It is easy to see
 that under the assumptions of Theorem \ref{WeakConvThm},
 by the law of the iterated logarithm, it holds
 $$
 \limsup_{n\to\infty} \frac{\sqrt{n}}{\sqrt{2\log\log n}}M_n=1.
 $$

 \section{Proofs}
 \label{sec.2}

 \begin{proof}[Proof of Theorem \ref{theo.2.1}]
 (a) For each $n\in\mathbb{N}^+$ consider the random variable
 \[
 Z_n:=\max\left\{\bar{X}_1,\bar{X}_2,\ldots,\bar{X}_n\right\}
 \]
 and call $F_n$ its distribution function. The sequence $(Z_n)_{n\ge1}$ is
 increasing and $Z_\infty=\lim_{n\to\infty} Z_n$,
 $F_\infty(x)=\lim_{n\to\infty} F_n(x)$.
 We will compute $F_n$ recursively.

 For $n\in \mathbb{N}^+$ and $x\geq 0$ we have
 {\small
 \begin{equation*}
 \begin{aligned}
 F_{n+1}(x) &=
 \Pr[X_1\leq x, X_1+X_2\leq 2x, \ldots,
 X_1+X_2+\cdots+X_{n+1}\leq (n+1) x]
 \\ & =
 \int_0^{x}\int_{0}^{2x-y_1}
 \cdots\int_{0}^{(n+1)x-(y_1+y_2+\cdots+y_n)}
 e^{-(y_1+y_2+\cdots+y_{n+1})}
 d {\mathbf{y}}_{n+1}
 \\
 &=\int_0^{x}\int_{0}^{2x-y_1}
 \cdots\int_{0}^{nx-(y_1+y_2+\cdots+y_{n-1})}
 \Big\{e^{-(y_1+y_2+\cdots+y_n)}-e^{-(n+1)x}\Big\}
 d {\mathbf{y}}_n
 \\&=F_n(x)-e^{-(n+1)x} \Vol(K_n(x))
 \end{aligned}
 \end{equation*}
 }
 where $d {\mathbf{y}}_k=dy_k \cdots d y_2 d y_1$ and
 \[K_n(x):=
 \{(y_1,y_2,\ldots,y_n) \in \mathbb{R}^n_+
 : \
 0\leq y_1+\cdots+y_i \leq ix, \ \
 i=1,2,\ldots,n\}.
 \]

 Note that $F_1(x)=1-e^{-x}$
 and introduce the convention $\Vol(K_0(x))=1$.
 It follows that
 $F_n(x)=1-\sum_{k=1}^n \Vol(K_{k-1}(x)) e^{-kx}$
 and  from Lemma \ref{VolumeLemma}, below,
 we get the explicit form
 \begin{equation*}
 F_n(x)=1-\sum_{k=1}^n \frac{k^{k-1}}{k!} x^{k-1}e^{-kx}, \text{ for all }
 \ x\geq 0, \ n\in \mathbb{N}^+.
 \end{equation*}
 This implies the first formula for $F_\infty$. By the law of large numbers,
 we get that $F_\infty(x)=0$ for all $x\in(-\infty, 1)$, and thus, the
 derivative of $F_\infty$ in $\mathbb{R}\sm\{1\}$ is
 \begin{equation*}
 f_\infty(x):=\textbf{1}_{x>1}\sum_{k=1}^\infty
 \frac{k^{k-1}}{k!}\left(k-\frac{k-1}{x}\right) x^{k-1} e^{-kx}.
 \end{equation*}

 (b) First we rewrite $F_\infty$ in a more convenient form.
 The fact that $F_\infty(x)=0$ for $x\in[0, 1)$ implies the
 remarkable identity (see Fig.\ 1)
 \begin{equation}
 \label{RemarkableEquality}
 \sum_{k=1}^{\infty} \frac{k^{k-1}}{k!} x^{k-1} e^{-kx} = 1 \ \
 \text{ for all } x\in[0,1).
 \end{equation}
 \begin{figure}[htp]
\vspace*{-7em}
\label{fig.1}
    \centering
     \includegraphics[width=12cm]
{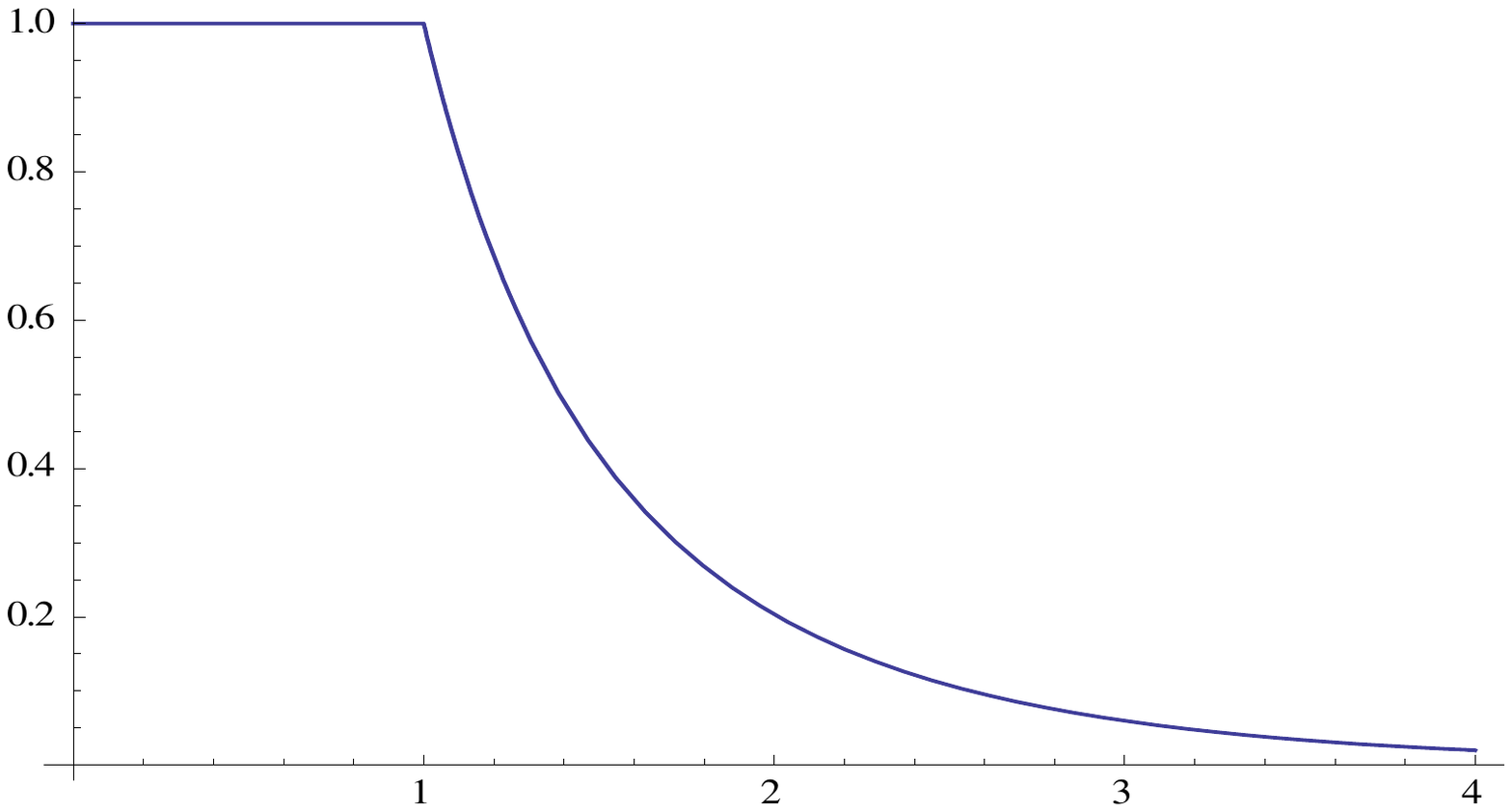}
\vspace*{-31em}

  \caption{The series \eqref{RemarkableEquality} in the interval $0\leq x\leq 4$.}
 \end{figure}
 \vspace*{-3ex}

 \noindent
 Our aim is to compute the value of the series in the left hand side also for $x\ge1$.
 The series converges uniformly for $x\in[0, \infty)$ because
 $$
 \sup_{x\ge0}  \frac{k^{k-1}}{k!} x^{k-1} e^{-kx}= \frac{(k-1)^{k-1}}{k!} e^{-(k-1)}
 \sim \frac{1}{k^{3/2} \sqrt{2\pi}},
 $$
 which is summable in $k$.
 Thus, by continuity,
 \eqref{RemarkableEquality} holds also for $x=1$. Now we rewrite \eqref{RemarkableEquality}
 in the form
 \begin{equation} \label{InverseFunction}
 \sum_{k=1}^{\infty}\frac{k^{k-1}}{k!} (x e^{-x})^k=x \text { for all } x\in[0, 1].
 \end{equation}
 The power series
 $h(y):=\sum_{k=1}^{\infty}\frac{k^{k-1}}{k!} y^k$ is strictly increasing in
 $[0, e^{-1}]$ and thus \eqref{InverseFunction} says that $h$ is the inverse
 function of the restriction, $g_r$, on $[0, 1]$ of the function
 $g:[0, \infty)\to[0, e^{-1}]$ with $g(x)=xe^{-x}$.
 The function $g$ is continuous, strictly increasing in $[0, 1]$, and
 strictly decreasing in $[1, \infty)$
 with $g(0)=0, g(1)=e^{-1}, g(\infty)=0$. Thus, for each $x\in[1,\infty)$,
 there exists a unique
 $t=t(x)\in(0,1]$ such that $g_r(t)=xe^{-x}$, i.e.,
 $te^{-t}=xe^{-x}$; hence, we define
 \begin{equation}
 \label{t}
 t(x):=g_r^{-1}(xe^{-x})=h(x e^{-x}),  \ \ x\geq 0.
 \end{equation}
 Since $t(x)=x$ for $x\in[0,1]$, we have
 \begin{equation}
 \label{2.7}
 F_{\infty}(x)=
 \left\{
 \begin{array}{lll}
 0, & \mbox{if} & x\leq 1,
 \\
 1-\frac{t(x)}{x}, & \mbox{if} & x\geq 1.
 \end{array}
 \right.
 \end{equation}

 Now for any fixed $u\in(0,1)$, the relation $F_\infty(x)=u$ gives
 $x-t(x)=xu$ so that $t(x)=(1-u)x$. Consequently,
 \begin{equation*}
 e^{xu}=\frac{e^{-t(x)}}{e^{-x}}=\frac{x}{t(x)}=\frac{1}{1-u}.
 \end{equation*}
 Thus, $x=-\log(1-u)/u$ and the proof is complete.
 \end{proof}
 \begin{remark}
 From the well-known relation $\E Z_n^\alpha =\alpha
 \int_{0}^{\infty}x^{\alpha-1}(1-F_n(x))dx$ for
 $\alpha>0$,  we obtain a simple expression
 for the moments:
 \begin{equation*}
 \E Z_n^\alpha =\alpha\sum_{k=1}^n
 \frac{\Gamma(\alpha+k-1)}{k^\alpha
 k!}.
 \end{equation*}
 In particular,
 \[
 \mbox{
 $\E Z_n=\sum_{k=1}^n \frac{1}{k^2}$, \ \
 $\E Z_n^2 =2\sum_{k=1}^n \frac{1}{k^2}$, \ \
 $\E Z_n^3= 3\sum_{k=1}^n \frac{1}{k^2}+3\sum_{k=1}^n \frac{1}{k^3}$.
 }
 \]
 Since $Z_n \nearrow Z_\infty$ with probability one, the above relations
 combined with the monotone convergence theorem give the moments of $Z_\infty$
 and in particular that it has mean
 $\frac{\pi^2}{6}$ and variance $\frac{\pi^2}{6}(2-\frac{\pi^2}{6})$.
 \end{remark}

 The next lemma is a special case of Theorem 1 in \cite{PS02} (see relation (7) in that paper),
 however, to keep the exposition self-contained, we provide a proof.

 \begin{lemma}
 \label{VolumeLemma}
 For $x\geq 0$, $x+t\geq 0$, and $n\in\mathbb{N}^+$, define
 \[
 K_n(x, t):=
 \{ (y_1,y_2,\ldots,y_n) \in \mathbb{R}^n_+
 : y_1+\cdots+y_i \leq ix+t \ \text{ for all }
 i=1,2,\ldots,n\}.
 \]
 Then,
 \begin{equation}
 \label{2.1}
 V_n(x, t):=\Vol(K_n(x, t))=\frac{1}{n!}(x+t)((n+1)x+t)^{n-1}, \ \
 n=1,2,\ldots,
 \end{equation}
 and, in particular, setting $t=0$,
 $\Vol(K_n(x))=\frac{1}{n!}(n+1)^{n-1} x^n$.
 \end{lemma}

 \begin{proof}
 Clearly $V_1(x,t)=x+t$ and for $n\geq 1$
 \begin{equation}
 \begin{aligned} \label{RecFormula}
 V_{n+1}(x, t)&=\int_0^{x+t}\int_{0}^{2x+t-y_1}
 \cdots\int_{0}^{(n+1)x+t-(y_1+y_2+\cdots+y_{n})}
 d\mathbf{y}_{n+1}
 %dy_{n+1}\cdots dy_{3}dy_2dy_1
 \\
 &=\int_0^{x+t}\int_{0}^{x+(x+t-y_1)}
 \cdots\int_{0}^{nx+(x+t-y_1)-(y_2+\cdots+y_{n})}
 d\mathbf{y}_{n+1}
 %dy_{n+1}\cdots dy_{3}dy_2dy_1
 \\
 &=\int_0^{x+t} V_n(x,x+t-y_1) dy_1.
 \end{aligned}
 \end{equation}
 The claim
 follows by induction on $n$.
 \end{proof}

 It is consistent with the recursion \eqref{RecFormula} for
 $V_n$ and \eqref{2.1} to define $V_{0}(x,t):=1$ so that
 \eqref{2.1} holds for all $n\in \mathbb{N}^+\cup\{0\}$.
  %nonnegative integers.
 This agrees
 with the convention $\Vol(K_0(x))=1$ we made in the proof of
 Theorem \ref{theo.2.1}(a).

 % Note that for $x\in[0,\frac1n]$, $V_n(x)=\Pr[\max\limits_{1\leq i\leq
 %n}\{\frac{U_1+\cdots+U_i}{i}\}\leq x]$ when
 %$U_i\sim U(0,1)$.

 \begin{proof}[Proof of Theorem \ref{WeakConvThm}]
 By Theorem 2.2.4 in \cite{CR81} we
 may assume that we can place $(X_i)_{i\ge1}$ in the same probability space with
 a standard Brownian motion $(W_s)_{s\ge0}$, so that, with probability 1,
 we have
 $|n \bar{X}_n-W_n|/n^{1/p}(\log n)^{1/2}\to0$ as $n\to\infty$. This implies that
 \begin{equation*}
 \lim_{n\to\infty}\sqrt{n}
 \left(M_n-\sup_{k\in \mathbb{N}, k\ge n} \frac{W_k}{k}\right)=0
 \end{equation*}
 with probability 1. On the other hand, with probability one, we have for
 all large $n$ the bound $\sup_{s\in[n, n+1]}|W_s-W_n|\le 2\sqrt{\log n}$, thus
 \begin{equation*}
 \lim_{n\to\infty}\sqrt{n}
 \left(\sup_{k\in \mathbb{N}, k\ge n} \frac{W_k}{k}-\sup_{s\ge n}
 \frac{W_s}{s}\right)=0.
 \end{equation*}
 Finally, by scaling and time inversion, we conclude that
 \begin{equation*}
 \sqrt{n}\sup_{s\ge n} \frac{W_s}{s}\overset{d}{=}\sup_{s\ge 1}
 \frac{W_s}{s}\overset{d}{=}\sup_{s\in[0, 1]} W_s \overset{d}{=} |W_1|,
 \end{equation*}
 and the proof is complete.
 \end{proof}

 \section{An application to ruin probability}
Following the same steps as in the proof of Theorem \ref{theo.2.1}(b), one can
 evaluate the distribution function, $F_{n;\lambda}$, of the random variable
 \[
 Z_{n;\lambda}:=
 \max\left\{\frac {X_1}{1+\lambda},\frac{X_1+X_2}{2+\lambda},
 \ldots,\frac{X_1+X_2+\cdots+X_n}{n+\lambda}\right\}
 \]
 for all $\lambda> -1$ and $n\in\mathbb{N}^+$.
 Indeed, using  \eqref{2.1} and induction on $n$
 it is easily verified that for all $x\ge0$ we have
\begin{equation*}
  F_{n;\lambda}(x)=1-(1+\lambda)e^{-\lambda x}\sum_{k=1}^n
  \frac{k (k+\lambda)^{k-2}}{k!} x^{k-1}e^{-kx}.
 \end{equation*}
 Thus, the distribution function of
 $Z_{\infty, \lambda}:=\lim_{n\to\infty} Z_{n; \lambda}$ equals
 \begin{align}
  F_{\infty;\lambda}(x)&=1-(1+\lambda)e^{-\lambda x}\sum_{k=1}^\infty
  \frac{k (k+\lambda)^{k-2}}{k!} x^{k-1}e^{-kx}
  \label{s1}
  \\&=1-\frac{t(x)}{x} e^{\lambda(t(x)-x)},
  \label{f2}
 \end{align}
 where the function $t$ is defined by \eqref{t}.
 %the proof of Theorem \ref{theo.2.1}(b).
 To justify the equality \eqref{f2}, we use the same arguments that lead from
 \eqref{RemarkableEquality} to \eqref{2.7}. Similarly as in Theorem \ref{theo.2.1}(b),
 we find that  $F_{\infty; \lambda}$ is zero in
 $(-\infty, 1]$, strictly increasing in $[1, \infty)$ with range $[0, 1)$, and
 its distribution inverse
 is given by
 \begin{equation}
 \label{useful}
 F_{\infty; \lambda}^{-1}(u)
 =\frac{-\log(1-u)}{1-(1-u)^\frac{1}{1+\lambda}}\times \frac{1}{\lambda+1},
 \ \ 0<u<1.
 \end{equation}

 \begin{remark}
 By the law of large numbers, the series in the right hand side of \eqref{s1}
 equals to one for all $x\in[0,1]$. Therefore, setting $x=\alpha$, $1+\lambda=\theta$
 and $k\to k+1$, the function
 \begin{equation*}
 p(k;\alpha,\theta)=\theta e^{-\alpha (\theta+k)}
 \frac{\alpha^{k}(k+\theta)^{k-1}}{k!}
 \end{equation*}
 defines a probability mass
 function supported on $\mathbb{N}^+\cup\{0\}$,
 known (after a suitable re-parametrization) as {\it
 generalized Poisson distribution}
 with parameter $(\alpha,\theta)\in[0,1]\times (0,\infty)$;
 see \cite{Char}
 and references therein.
 \end{remark}

 Consider now the following risk model. Assume that the aggregate claim at time $n$
 is described by $S_n:=X_1+\cdots+X_n$, where the $(X_i)_{i\ge1}$ are i.i.d. with $\E X_1=1$, the premium rate (per time unit) is
 $c=1+\theta>0$ ($\theta$ is the safety loading of the insurance), and the
 initial capital is $u>-(1+\theta)$, where negative initial capital
 is allowed for technical reasons.
 The risk process is defined by
 \begin{equation*}
 U_n=u + c n-S_n, \ \ n\in\mathbb{N}^+.
 \end{equation*}
 Clearly, the ruin probability
 \begin{equation}
 \label{ruin}
 \psi(u):=\Pr (U_n<0 \ \mbox{ for some } \  n\in\mathbb{N}^+)
 \end{equation}
 is of fundamental importance. Our explicit formulae are useful in computing
 the minimum initial capital needed to ensure that $\psi(u)$ is small.
 The particular problem (for general claims) has been studied in \cite{SSK},
 under the name {\it discrete-time surplus-process model}.
 It is well-known that  $\psi(u)=1$ when $c\leq 1$, no matter
 how large $u$ is, because $\E X_i=1$. Hence, the problem
 is meaningful only for $c>1$, i.e., $\theta>0$.

 \begin{theorem}
 \label{theo.3}
 Assume that the i.i.d.\ individual claims $(X_i)_{i\geq 1}$
 are exponential random variables with mean 1,
 fix $\alpha\in(0,1)$ and $\theta>0$,
 and set $c=1+\theta$.
 Then,

 \noindent
 {\rm (a)} the ruin probability \eqref{ruin}
 is given by
 \begin{equation}
 \label{psi}
 \psi(u)=\begin{cases} \frac{t(c)}{c}\exp\left(-u\left(1-\frac{t(c)}{c}\right)\right), &\text{ if }
 u>-c,\\
 1& \text{ if } u\le -c,
 \end{cases}
 \end{equation}
 where the function $t$ is given by \eqref{t};

 \noindent
 {\rm (b)} the minimum initial capital $u=u(\alpha,\theta)$ needed to ensure that
 $\psi(u)\leq \alpha$
 is given by the unique root of the equation
 \begin{equation}
 \label{u}
 (1+\theta+u)
 \left(1-\alpha^{\frac{1+\theta}{1+\theta+u}}\right)
 =-\log\alpha, \ \  u>-(1+\theta).
 \end{equation}

 \end{theorem}

 \begin{proof}
 (a) For $u>-c$, we can use \eqref{f2} to get
$$\psi(u)=1-F_{\infty; u/c}(c)=\frac{t(c)}{c} e^{(u/c)(t(c)-c)},$$
which is \eqref{psi}. Then, the definition of $t$ shows that $\lim_{u\to-c^+}\psi(u)=\frac{t(c)e^{-t(c)}}{ce^{-c}}=1$, and the monotonicity of $\psi$ implies that $\psi(u)=1$ for $u\le -c$.

 (b) By the formula of part (a), the function $\psi$ is strictly decreasing in the interval $(-c, \infty)$ and maps that interval to $(0, 1)$. Therefore, there is a unique $u=u(\alpha, \theta)>-c$ such that $\psi(u)=\alpha$. Let $\lambda:=u/c$, which is greater than $-1$.  Then,  using \eqref{useful}, we see that
 $$\psi(u)=\alpha \Leftrightarrow F_{\infty;\lambda}(c)=1-\alpha \Leftrightarrow c=F_{\infty;\lambda}^{-1}(1-\alpha)=\frac{-\log \alpha}{(1+\lambda)\Big(1-\alpha^\frac{1}{1+\lambda}\Big)}.
 $$
We substitute $c=1+\theta, \lambda=u/(1+\theta)$, and the above equivalences show that $u$ is the unique solution of
 \begin{equation*}
 \left(1+\frac{u}{1+\theta}\right)\left(1-\alpha^{\frac{1+\theta}{1+\theta+u}}\right)=\frac{-\log \alpha}{1+\theta}.
 \end{equation*}

 \end{proof}

 The exact values of $u$ in \eqref{u} are in perfect agreement with
 the numerical
 approximations given in the last line of Table 1 in \cite{SSK}.
 Notice that the initial capital $u$ can be negative sometimes, e.g.,
 $u(.5,.5)\simeq -.3107$.

\end{document}